\documentclass[11pt]{amsart}

\usepackage[T1]{fontenc}
\usepackage[utf8]{inputenc}
\usepackage{amsmath,amssymb,amsthm,mathtools}
\usepackage{hyperref}
\usepackage{enumitem}
\usepackage{mathrsfs}

\hypersetup{
  colorlinks=true,
  linkcolor=blue,
  citecolor=blue,
  urlcolor=blue
}

\newtheorem{theorem}{Theorem}[section]
\newtheorem{proposition}[theorem]{Proposition}
\newtheorem{remark}[theorem]{Remark}
\newtheorem*{conjecture}{Conjecture}

\newcommand{\R}{\mathbb{R}}
\newcommand{\Ric}{\operatorname{Ric}}

\begin{document}

\title[ACS condition on minimal isoparametric hypersurfaces]{ACS CONDITION ON MINIMAL ISOPARAMETRIC HYPERSURFACES OF POSITIVE RICCI CURVATURE IN UNIT SPHERES}

\author{Niang Chen}
\address{Faculty of Arts and Sciences, Beijing Normal University, Zhuhai 519087, China}
\email{chenniang@bnu.edu.cn}


\begin{abstract}
We study the Ambrozio--Carlotto--Sharp (ACS) criterion on minimal isoparametric hypersurfaces $N^{n+1}\subset S^{n+2}$ with positive Ricci curvature, motivated by the Schoen--Marques--Neves conjecture on Morse index.
For $g=4$ distinct principal curvatures with multiplicities $m_1,m_2$, we prove that the pointwise ACS inequality holds if and only if $\min\{m_1,m_2\}\ge 4$. Sufficiency is obtained via a moment-relaxation technique yielding the sharp bound $4a^2$ on the quadratic part of the integrand; necessity follows from an explicit extremal configuration in the top eigenspace of the shape operator. We also verify the ACS condition for $g=3$ with $m_1=m_2\in\{4,8\}$.
As a consequence, for any closed embedded minimal hypersurface $M^n$ in such an ambient manifold, $\operatorname{index}(M)\ge \tfrac{2}{d(d-1)}\, b_1(M)$ with $d=n+3$.
\end{abstract}

\maketitle

\section{Introduction}

In 1970, Lawson \cite{Lawson} constructed minimal surfaces in $S^3$ of arbitrary genus. This work initiated a new direction in differential geometry by moving the study of minimal surfaces from Euclidean space to positively curved ambient spaces. Later, Hsiang and Lawson \cite{HsiangLawson} extended related techniques to higher-dimensional spheres. On the other hand, Lawson's work also clarified topological restrictions for minimal surfaces in positively curved manifolds and influenced the study of incompressible surfaces in three-manifolds by Schoen and Yau \cite{SchoenYau}.

Besides the existence problem, an equally important question is to understand the geometric and topological properties of minimal hypersurfaces. Among these, the Morse index is a central analytical invariant: it measures stability and reflects deep information about the geometry of the hypersurface. It has been conjectured that the index controls the basic topology of minimal hypersurfaces in an effective way.

\begin{conjecture}[Schoen-Marques-Neves, see \cite{MarquesNeves}]
Let $(N^{n+1},g)$ be a closed Riemannian manifold with positive Ricci curvature. Then there exists a positive constant $C$ such that, for every closed embedded minimal hypersurface $M^n\subset N^{n+1}$,
\[
\operatorname{index}(M)\ge C\, b_1(M).
\]
\end{conjecture}

A particular ambient manifold for which this conjecture is known is the round sphere. This was proved by Savo \cite{Savo}, building on earlier work of Ros \cite{Ros}. For flat three-dimensional tori, Ros proved that the index of a closed minimal surface is bounded from below by an affine function of the genus. Since, for a closed orientable surface, the first Betti number is twice the genus, this is a close analogue of the conjectural picture above.

Ambrozio, Carlotto and Sharp \cite{ACS1} developed a general framework to study this conjecture in other positively curved ambient manifolds. Their method gives a lower bound for the Morse index in terms of the first Betti number under a suitable extrinsic inequality.

\begin{theorem}[Ambrozio-Carlotto-Sharp \cite{ACS1}]\label{thm:ACS}
Let $(N^{n+1},g)$ be a Riemannian manifold isometrically embedded in some Euclidean space $\R^d$. Let $M^n$ be a closed embedded minimal hypersurface of $N^{n+1}$. Assume that for every nonzero vector field $X$ on $M$,
\begin{equation}\label{eq:ACS}
\begin{aligned}
&\int_M \Bigg[\sum_{k=1}^n R^N(e_k,X,e_k,X)+\Ric_N(\nu,\nu)|X|^2\Bigg] \, dM \\
>
&\int_M \Bigg[\sum_{k=1}^n |\Pi(e_k,X)|^2 
+\sum_{k=1}^n |\Pi(e_k,\nu)|^2|X|^2\Bigg] \, dM,
\end{aligned}
\end{equation}
where $R^N$ is the curvature tensor of $N$, $\Pi$ is the second fundamental form of the embedding $N\hookrightarrow \R^d$, $\nu$ is a local unit normal vector field along $M$, and $\{e_k\}_{k=1}^n$ is a local orthonormal frame on~$M$.
Then
\[
\operatorname{index}(M)\ge \frac{2}{d(d-1)}\, b_1(M).
\]
\end{theorem}

For convenience, we refer to \eqref{eq:ACS} as the \emph{ACS condition}. The same strategy has since been extended to free boundary and weighted settings; see \cite{ACS2,ChenGeZhang,ImperaRimoldi,ImperaRimoldiSavo}.

In this paper we apply the ACS framework to minimal isoparametric hypersurfaces with positive Ricci curvature in the unit sphere. We hope to provide further evidence for the Schoen-Marques-Neves conjecture by adding new examples in this class.

For an isoparametric hypersurface $N^{n+1}\subset S^{n+2}$, the number $g$ of distinct principal curvatures can only be $1,2,3,4$ or $6$, and the multiplicities satisfy $m_i=m_{i+2}$ (indices modulo $g$). Hence there are at most two distinct multiplicities, say $m_1$ and $m_2$; see \cite{CecilRyan,GeQianTangYanOverview}. For a \emph{minimal} isoparametric hypersurface $N^{n+1}\subset S^{n+2}$, the cases with positive Ricci curvature are the following \cite{QianTangYan}:
\begin{enumerate}[label=(\arabic*)]
  \item when $g=1$, the Ricci curvature is obviously positive;
  \item when $g=2$, the Ricci curvature is positive unless $N$ is $S^1(r_1)\times S^n(r_2)$;
  \item when $g=3$, the multiplicities are equal and the Ricci curvature is positive for $m_1=m_2>1$;
  \item when $g=4$, the Ricci curvature is positive provided $m_1,m_2\ge 2$.
\end{enumerate}

Our main result is the following.

\begin{theorem}\label{thm:main}
Let $N^{n+1}\subset S^{n+2}$ be a minimal isoparametric hypersurface. Then the following statements hold.
\begin{enumerate}[label=(\arabic*)]
  \item If $g=3$ and $m_1=m_2\in\{4,8\}$, then $N$ satisfies the ACS condition.
  \item If $g=4$ and
  $
  \min\{m_1,m_2\}\ge 4,
 $
  then $N$ satisfies the ACS condition.
\end{enumerate}
Consequently, if $M^n$ is a closed minimal hypersurface of any of the above ambient manifolds, then
\[
\operatorname{index}(M)\ge \frac{2}{d(d-1)}\, b_1(M),
\]
where $d=n+3$ for the natural embedding $N\subset S^{n+2}\subset \R^{n+3}$.
\end{theorem}

Gorodski, Mendes and Radeschi \cite{GMR} also produced many examples supporting the Schoen-Marques-Neves conjecture. They obtained robust index lower bounds for minimal hypersurfaces in several isoparametric and symmetric ambient spaces by a different method.
Our contribution is instead to verify a sufficient ACS inequality for the ambient manifold itself. In the $g=4$ case, we work with the pointwise quantity $\Delta(X,\nu)$ and obtain progressively tighter bounds. A direct estimate of $\Delta$ using only the largest absolute curvature yields the range $\min\{m_1,m_2\}\ge 5$. By retaining a quadratic term in Proposition~\ref{prop:Delta} and completing the square, we lower the coefficient to $\tfrac{17}{4}$, which extends the result to $\min\{m_1,m_2\}=4$ with $m_1+m_2\le 19$. Finally, a moment-relaxation technique (Proposition~\ref{prop:g4opt}) that exploits the asymmetry between the largest and smallest principal curvatures produces the sharp coefficient $4$, establishing the ACS condition for all pairs with $\min\{m_1,m_2\}\ge 4$ without restriction on~$m_1+m_2$.
We observe in Remark~\ref{rem:pointwise-threshold} that $\min\{m_1,m_2\}\ge 4$ is also a \emph{necessary} condition for any argument based solely on pointwise positivity of~$\Delta$. Thus our result is sharp with respect to this method.

The cases not covered by Theorem \ref{thm:main} remain open from the point of view of the present computation. More precisely:
\begin{enumerate}[label=(\arabic*)]
  \item for $g=3$, the case $m_1=m_2=2$ is not decided by our argument, while for $m_1=m_2=1$ the Ricci curvature is not positive;
  \item for $g=4$, our estimate does not decide the ACS condition when at least one multiplicity lies in $\{2,3\}$;
  \item for $g=6$, the Ricci curvature is not positive.
\end{enumerate}

\section{Preliminaries}

An isoparametric hypersurface is a hypersurface that arises as a regular level set of an isoparametric function. More precisely, a hypersurface $N^{n+1}$ in a Riemannian manifold $\widetilde N^{n+2}$ is called \emph{isoparametric} if there exists a nonconstant smooth function $f$ on $\widetilde N$ and smooth functions $a,b$ such that
\begin{enumerate}[label=(\arabic*)]
  \item $N=f^{-1}(c)$ for some regular value $c$ of $f$;
  \item $|\nabla f|^2=b(f)$;
  \item $\Delta f=a(f)$.
\end{enumerate}
In a space form, this is equivalent to requiring that all principal curvatures are constant.

For isoparametric hypersurfaces in the unit sphere $S^{n+2}$, the possible numbers of distinct principal curvatures are $1,2,3,4$ and $6$. The cases may be summarized as follows (see \cite{Chi,CecilRyan,MuenznerI,MuenznerII,GeQianTangYanOverview}):
\begin{enumerate}[label=(\arabic*)]
  \item if $g=1$, then $N$ is totally umbilic;
  \item if $g=2$, then $N$ is a Clifford hypersurface $S^k(r)\times S^{n+1-k}(\sqrt{1-r^2})$;
  \item if $g=3$, then all multiplicities are equal and can be $1,2,4$ or $8$;
  \item if $g=6$, then all multiplicities are equal and can be $1$ or $2$;
  \item if $g=4$, then the examples are either homogeneous or of OT-FKM type; among the multiplicity pairs, only $(2,2)$ and $(4,5)$ are not of OT-FKM type.
\end{enumerate}
For OT-FKM hypersurfaces, the multiplicities are of the form $(m_1,m_2)=(m_1,\ell-m_1-1)$ with $\ell\in \mathbb N^*$.

We will also use the following result of M\"unzner.

\begin{theorem}[M\"unzner, see \cite{Chi,MuenznerI,MuenznerII}]\label{thm:Munzner}
Let $M^n$ be an isoparametric hypersurface with $g$ distinct principal curvatures in $S^{n+1}$. Let
\[
\lambda_1=\cot\theta_1>\cdots>\lambda_g=\cot\theta_g,
\qquad \theta_i\in(0,\pi),
\]
be the principal curvatures, with multiplicities $m_1,\dots,m_g$. Then:
\begin{enumerate}[label=(\arabic*)]
  \item there exists a homogeneous polynomial $F$ of degree $g$ on $\R^{n+2}$, called the Cartan-M\"unzner polynomial, such that
  \[
  |\nabla F|^2=g^2 r^{2g-2},
  \qquad
  \Delta F=\frac{m_2-m_1}{2}\, g^2 r^{g-2},
  \]
  where $r$ is the radial function on $\R^{n+2}$;
  \item if $f=F|_{S^{n+1}}$, then the range of $f$ is $[-1,1]$, the only critical values are $\pm1$, and the critical sets $M_\pm=f^{-1}(\pm1)$ are connected focal submanifolds;
  \item for each $c\in(-1,1)$, the level set $M_c=f^{-1}(c)$ is an isoparametric hypersurface, the multiplicities satisfy $m_i=m_{i+2}$ (indices modulo $g$), and
  \[
  \theta_i=\theta_1+\frac{i-1}{g}\pi;
  \]
  in particular, if $g$ is odd then all multiplicities are equal;
  \item the unique minimal member in the isoparametric family is $M_{c_0}=f^{-1}(c_0)$, where
  \[
  c_0=\frac{m_2-m_1}{m_1+m_2}.
  \]
\end{enumerate}
\end{theorem}

\section{The computations}

Let $N^{n+1}\subset S^{n+2}$ be a minimal isoparametric hypersurface. To check the ACS condition, it is enough to study the pointwise integrand in \eqref{eq:ACS}.

Fix a closed minimal hypersurface $M^n\subset N^{n+1}$. Let $\nu$ be a local unit normal vector field of $M$ in $N$, and let $\{e_k\}_{k=1}^n$ be a local orthonormal frame on $M$. We define
\begin{equation}\label{eq:DeltaDef}
\begin{aligned}
\Delta(X,\nu):={}&\sum_{k=1}^n R^N(e_k,X,e_k,X)+\Ric_N(\nu,\nu)|X|^2 \\
&-\sum_{k=1}^n |\Pi(e_k,X)|^2-\sum_{k=1}^n |\Pi(e_k,\nu)|^2|X|^2.
\end{aligned}
\end{equation}
If $\Delta(X,\nu)>0$ for every nonzero tangent vector $X$, then $N$ satisfies the ACS condition.

Let $\{\varepsilon_i\}_{i=1}^{n+1}$ be a local orthonormal frame of principal directions of $N\subset S^{n+2}$, with corresponding principal curvatures $\{\lambda_i\}_{i=1}^{n+1}$. Write
\[
X=\sum_{i=1}^{n+1} x_i\varepsilon_i,
\qquad
\nu=\sum_{i=1}^{n+1} \nu_i\varepsilon_i.
\]
Since \eqref{eq:DeltaDef} is homogeneous in $X$, we may assume $|X|=1$. We also have $|\nu|=1$ and $\langle X,\nu\rangle=0$.

\begin{proposition}\label{prop:Delta}
With the notation above,
\begin{align*}
\Delta(X,\nu)
={}&(2n-2)-2\sum_{i=1}^{n+1} x_i^2\lambda_i^2-2\sum_{i=1}^{n+1} \nu_i^2\lambda_i^2
+2\Big(\sum_{i=1}^{n+1} x_i\nu_i\lambda_i\Big)^2 \\
&\qquad
+\Big(\sum_{i=1}^{n+1} \nu_i^2\lambda_i\Big)^2
-\Big(\sum_{i=1}^{n+1} \nu_i^2\lambda_i\Big)
 \Big(\sum_{i=1}^{n+1} x_i^2\lambda_i\Big).
\end{align*}
\end{proposition}

\begin{proof}
Let $\eta$ be a local unit normal vector field of $N$ in $S^{n+2}$, and let $\xi$ be the unit normal vector field of $S^{n+2}$ in $\R^{n+3}$. Then, for tangent vectors $U,V\in TN$,
\[
\Pi(U,V)=\Pi^{\eta}(U,V)+\Pi^{\xi}(U,V)=\langle AU,V\rangle\eta+\langle U,V\rangle\xi,
\]
where $A$ is the shape operator of $N\subset S^{n+2}$. In the principal frame,
\[
A(\varepsilon_i)=\lambda_i\varepsilon_i,
\qquad
\Pi(\varepsilon_i,\varepsilon_j)=\lambda_i\delta_{ij}\eta+\delta_{ij}\xi.
\]

We first compute the curvature term. Since $\{e_1,\dots,e_n,\nu\}$ is an orthonormal frame of $TN$,
\[
\sum_{k=1}^n R^N(e_k,X,e_k,X)=\Ric_N(X,X)-R^N(\nu,X,\nu,X).
\]
By the Gauss equation for $N\subset S^{n+2}$,
\[
K(\varepsilon_i,\varepsilon_j)=1+\lambda_i\lambda_j
\qquad (i\neq j),
\]
and the minimality of $N$ implies $\sum_i \lambda_i=0$. Therefore,
\begin{align*}
\Ric_N(\varepsilon_i,\varepsilon_i)
&=\sum_{j\ne i} K(\varepsilon_j,\varepsilon_i)
=\sum_{j\ne i}(1+\lambda_j\lambda_i)
=n+\lambda_i\sum_{j\ne i}\lambda_j
=n-\lambda_i^2.
\end{align*}
Hence
\begin{equation}\label{eq:RicX}
\Ric_N(X,X)=\sum_{i=1}^{n+1}(n-\lambda_i^2)x_i^2.
\end{equation}
On the other hand,
\begin{align*}
\Pi(\nu,\nu)
&=\Big(\sum_{i=1}^{n+1}\nu_i^2\lambda_i\Big)\eta+\xi,
\\
\Pi(X,X)
&=\Big(\sum_{i=1}^{n+1}x_i^2\lambda_i\Big)\eta+\xi,
\\
\Pi(\nu,X)
&=\Big(\sum_{i=1}^{n+1}x_i\nu_i\lambda_i\Big)\eta,
\end{align*}
since $\langle \nu,X\rangle=0$. Applying the Gauss equation again, we obtain
\begin{align}
R^N(\nu,X,\nu,X)
&=\langle \Pi(\nu,\nu),\Pi(X,X)\rangle-|\Pi(\nu,X)|^2 \notag \\
&=\Big(\sum_{i=1}^{n+1}\nu_i^2\lambda_i\Big)
   \Big(\sum_{i=1}^{n+1}x_i^2\lambda_i\Big)
 +1
 -\Big(\sum_{i=1}^{n+1}x_i\nu_i\lambda_i\Big)^2.
\label{eq:Rnx}
\end{align}
Combining \eqref{eq:RicX} and \eqref{eq:Rnx}, we get
\begin{align}
\sum_{k=1}^n R^N(e_k,X,e_k,X)
={}&\sum_{i=1}^{n+1}(n-\lambda_i^2)x_i^2
-\Big(\sum_{i=1}^{n+1}\nu_i^2\lambda_i\Big)
 \Big(\sum_{i=1}^{n+1}x_i^2\lambda_i\Big) \notag \\
&\qquad -1
+\Big(\sum_{i=1}^{n+1}x_i\nu_i\lambda_i\Big)^2.
\label{eq:firstPart}
\end{align}
Moreover,
\begin{equation}\label{eq:RicNu}
\Ric_N(\nu,\nu)|X|^2=
\sum_{i=1}^{n+1}(n-\lambda_i^2)\nu_i^2.
\end{equation}

We next compute the second fundamental form terms. Since $\{e_1,\dots,e_n,\nu\}$ is an orthonormal basis of~$TN$,
\begin{align*}
\sum_{k=1}^n |\Pi(e_k,X)|^2
&=\sum_{i=1}^{n+1}|\Pi(\varepsilon_i,X)|^2-|\Pi(\nu,X)|^2,
\\
\sum_{k=1}^n |\Pi(e_k,\nu)|^2
&=\sum_{i=1}^{n+1}|\Pi(\varepsilon_i,\nu)|^2-|\Pi(\nu,\nu)|^2.
\end{align*}
Now,
\[
\Pi(\varepsilon_i,X)=x_i\lambda_i\eta+x_i\xi,
\qquad
\Pi(\varepsilon_i,\nu)=\nu_i\lambda_i\eta+\nu_i\xi,
\]
so that
\[
|\Pi(\varepsilon_i,X)|^2=x_i^2(\lambda_i^2+1),
\qquad
|\Pi(\varepsilon_i,\nu)|^2=\nu_i^2(\lambda_i^2+1).
\]
Therefore,
\begin{align}
\sum_{k=1}^n |\Pi(e_k,X)|^2
+\sum_{k=1}^n |\Pi(e_k,\nu)|^2|X|^2
={}&\sum_{i=1}^{n+1}x_i^2(\lambda_i^2+1)
+\sum_{i=1}^{n+1}\nu_i^2(\lambda_i^2+1) \notag \\
&-\Big(\sum_{i=1}^{n+1}x_i\nu_i\lambda_i\Big)^2 \notag \\
&-\Bigg[\Big(\sum_{i=1}^{n+1}\nu_i^2\lambda_i\Big)^2+1\Bigg].
\label{eq:secondPart}
\end{align}

Finally, substituting \eqref{eq:firstPart}, \eqref{eq:RicNu} and \eqref{eq:secondPart} into \eqref{eq:DeltaDef}, and using $\sum_i x_i^2=\sum_i \nu_i^2=1$, we obtain the claimed formula.
\end{proof}

\begin{proposition}\label{prop:g4opt}
Assume $g=4$, set
\[
q:=\min\{m_1,m_2\}=4,
\qquad
s:=m_1+m_2,
\]
and let
\[
a^2=\frac{1+\sqrt{1-4/s}}{1-\sqrt{1-4/s}}.
\]
If $s\ge 15$, then for every orthogonal unit vectors $X,\nu\in T_pN$ one has
\[
\Delta(X,\nu)\ge 4s-4-4a^2>0.
\]
In particular, $N$ satisfies the ACS condition.
\end{proposition}

\begin{proof}
After interchanging $m_1$ and $m_2$ if necessary, we may assume
\[
m_1=4\le m_2.
\]
For the minimal leaf with $g=4$, we have
\[
\cos 2\theta_1=\sqrt{\frac{m_2}{s}},
\qquad
\sin 2\theta_1=\sqrt{\frac{m_1}{s}}=\sqrt{\frac{4}{s}},
\]
and therefore
\[
\lambda_1^2
=\frac{1+\sqrt{m_2/s}}{1-\sqrt{m_2/s}},
\qquad
\lambda_4^2
=\frac{1+\sqrt{m_1/s}}{1-\sqrt{m_1/s}}.
\]
Since $m_2\ge m_1$ and the function
\[
t\longmapsto \frac{1+\sqrt t}{1-\sqrt t}
\]
is strictly increasing on $(0,1)$, it follows that $\lambda_1^2\ge \lambda_4^2$. By definition of $a$, we therefore have
\[
a=\lambda_1=\cot\theta_1>0.
\]
Moreover, $\lambda_1$ has multiplicity $m_1=4$, so the $a$-eigenspace has dimension exactly~$4$.

Now write $\lambda_4=-b<0$. Then
\[
b:=-\lambda_4
=-\cot\Bigl(\theta_1+\frac{3\pi}{4}\Bigr)
=\tan\Bigl(\theta_1+\frac{\pi}{4}\Bigr)
=\frac{a+1}{a-1}.
\]
Hence all principal curvatures lie in the interval $[-b,a]$.

Next we explain why $b\le a/2$ when $s\ge 15$. Set
\[
x(s):=\sqrt{1-\frac{4}{s}}\in(0,1).
\]
Then
\[
a^2=\frac{1+x(s)}{1-x(s)}.
\]
Since $x(s)$ is increasing in $s$ and the function $x\mapsto \frac{1+x}{1-x}$ is strictly increasing on $(0,1)$, the quantity $a^2$ is increasing in~$s$. Hence the assumption $s\ge 15$ gives
\[
a^2\ge \frac{1+\sqrt{11/15}}{1-\sqrt{11/15}}
=\frac{13+\sqrt{165}}{2}
>
\frac{13+3\sqrt{17}}{2}
=\left(\frac{3+\sqrt{17}}{2}\right)^2.
\]
Therefore
\[
a>\frac{3+\sqrt{17}}{2}.
\]
On the other hand, since $a>1$ and $b=\frac{a+1}{a-1}$, we have
\[
b\le \frac{a}{2}
\iff
\frac{a+1}{a-1}\le \frac{a}{2}
\iff
2(a+1)\le a(a-1)
\iff
a^2-3a-2\ge 0
\iff
a\ge \frac{3+\sqrt{17}}{2}.
\]
Thus the previous inequality for $a$ implies
\[
b\le \frac{a}{2}.
\]

Now define
\[
A:=\sum_{i=1}^{n+1}x_i^2\lambda_i,
\qquad
B:=\sum_{i=1}^{n+1}\nu_i^2\lambda_i,
\qquad
C:=\sum_{i=1}^{n+1}x_i\nu_i\lambda_i,
\]
and
\[
S_x:=\sum_{i=1}^{n+1}x_i^2\lambda_i^2,
\qquad
S_\nu:=\sum_{i=1}^{n+1}\nu_i^2\lambda_i^2.
\]
By Proposition \ref{prop:Delta},
\[
\Delta=(2n-2)-2S_x-2S_\nu+2C^2+B^2-AB.
\]
Equivalently,
\[
\Delta=(2n-2)-F+2C^2,
\qquad
F:=2S_x+2S_\nu-B^2+AB.
\]

Since each principal curvature $\lambda_i$ lies in $[-b,a]$, we have
\[
(\lambda_i-a)(\lambda_i+b)\le 0,
\]
hence
\[
\lambda_i^2\le (a-b)\lambda_i+ab.
\]
Averaging against the weights $x_i^2$ and $\nu_i^2$, we obtain
\[
S_x\le (a-b)A+ab,
\qquad
S_\nu\le (a-b)B+ab.
\]
Therefore
\[
F\le G(A,B),
\]
where
\[
G(A,B):=2(a-b)A+2(a-b)B+4ab-B^2+AB.
\]

Since $A,B\in[-b,a]$, we have
\[
\frac{\partial G}{\partial A}=2(a-b)+B\ge 2(a-b)-b=2a-3b\ge \frac{a}{2}>0,
\]
because $b\le a/2$. Thus $G$ is strictly increasing in $A$ throughout the square
$[-b,a]\times[-b,a]$, and so its maximum is attained at $A=a$. Hence
\[
G(A,B)\le G(a,B)
=-B^2+(3a-2b)B+2a^2+2ab.
\]
This is a concave quadratic in $B$, with axis
\[
B_0=\frac{3a-2b}{2}=\frac{3a}{2}-b\ge a,
\]
since $b\le a/2$. It follows that $G(a,B)$ is increasing on $[-b,a]$, and therefore
\[
F\le G(A,B)\le G(a,a)=4a^2.
\]
Consequently,
\[
\Delta(X,\nu)\ge (2n-2)-4a^2.
\]
Since $n=2s-1$, this becomes
\[
\Delta(X,\nu)\ge 4s-4-4a^2.
\]

Finally, let
\[
x:=\sqrt{1-\frac{4}{s}}\in(0,1).
\]
Then
\[
a^2=\frac{1+x}{1-x},
\qquad
s=\frac{4}{1-x^2},
\]
so
\[
s-1-a^2
=
\frac{3+x^2}{1-x^2}-\frac{1+x}{1-x}
=
\frac{2-2x}{1-x^2}
=
\frac{2}{1+x}>0.
\]
Hence
\[
4s-4-4a^2=4(s-1-a^2)>0,
\]
and therefore $\Delta(X,\nu)>0$ for every orthogonal unit pair $(X,\nu)$.

Moreover, this lower bound is sharp. Indeed, the $a$-eigenspace has dimension $q=4$, so we may choose orthogonal unit vectors $X,\nu$ in that eigenspace. Then
\[
A=B=a,\qquad S_x=S_\nu=a^2,\qquad C=0,
\]
and Proposition \ref{prop:Delta} gives
\[
\Delta(X,\nu)=4s-4-4a^2.
\]
\end{proof}

\begin{proof}[Proof of Theorem \ref{thm:main}]
For simplicity, write
\[
A:=\sum_{i=1}^{n+1} x_i^2\lambda_i,
\qquad
B:=\sum_{i=1}^{n+1} \nu_i^2\lambda_i,
\qquad
C:=\sum_{i=1}^{n+1} x_i\nu_i\lambda_i,
\]
and
\[
S_x:=\sum_{i=1}^{n+1} x_i^2\lambda_i^2,
\qquad
S_\nu:=\sum_{i=1}^{n+1} \nu_i^2\lambda_i^2.
\]
Then Proposition \ref{prop:Delta} gives
\[
\Delta=(2n-2)-2S_x-2S_\nu+2C^2+B^2-AB.
\]
It suffices to show that $\Delta>0$.

\medskip
\noindent\textbf{(1) The case $g=3$.}
For a minimal isoparametric hypersurface with $g=3$, the principal curvatures are
\[
\lambda_1=\sqrt{3},\qquad \lambda_2=0,\qquad \lambda_3=-\sqrt{3}.
\]
If the common multiplicity is $m\in\{4,8\}$, then
\[
n=3m-1\ge 11,
\qquad
2n-2\ge 20.
\]
Since $|X|=|\nu|=1$ and $\langle X,\nu\rangle=0$, we have
\[
\sum_i x_i^2=\sum_i \nu_i^2=1,
\qquad
\sum_i x_i\nu_i=0.
\]
Hence
\[
\sum_i x_i^2\lambda_i^2\le 3,
\qquad
\sum_i \nu_i^2\lambda_i^2\le 3,
\qquad
\left|\Big(\sum_i \nu_i^2\lambda_i\Big)\Big(\sum_i x_i^2\lambda_i\Big)\right|\le 3.
\]
Discarding the nonnegative terms in the expression for $\Delta$, we obtain
\[
\Delta\ge (2n-2)-6-6-3\ge 20-6-6-3>0.
\]
Therefore $N$ satisfies the ACS condition when $g=3$ and $m\in\{4,8\}$.

\medskip
\noindent\textbf{(2) The case $g=4$.}
Set
\[
s:=m_1+m_2,
\qquad
n=2s-1.
\]
By Theorem \ref{thm:Munzner}, the principal curvatures are
\[
\lambda_i=\cot\Big(\theta_1+\frac{i-1}{4}\pi\Big),
\qquad i=1,2,3,4,
\]
where, for the minimal leaf,
\[
\cos 4\theta_1=c_0=\frac{m_2-m_1}{m_1+m_2}.
\]
Let
\[
a^2:=\max_{1\le i\le 4}\lambda_i^2=\max\{\lambda_1^2,\lambda_4^2\}.
\]
Since
\[
S_x\le a^2,
\qquad
S_\nu\le a^2,
\qquad
|A|\le a,
\qquad
|B|\le a,
\]
we first record the rough estimate obtained by discarding the nonnegative terms $2C^2$ and $B^2$ and bounding $|AB|\le a^2$:
\begin{equation}\label{eq:roughg4}
\Delta\ge (2n-2)-5a^2.
\end{equation}
Thus it is enough to require
\begin{equation}\label{eq:aineq}
a^2<\frac{2n-2}{5}=\frac{4}{5}(s-1).
\end{equation}

We do not assume any ordering between $m_1$ and $m_2$. Since $0<2\theta_1<\pi/2$, we have
\[
\cos 2\theta_1=\sqrt{\frac{1+\cos 4\theta_1}{2}}=\sqrt{\frac{m_2}{s}},
\qquad
\sin 2\theta_1=\sqrt{\frac{1-\cos 4\theta_1}{2}}=\sqrt{\frac{m_1}{s}}.
\]
Therefore,
\begin{align*}
\lambda_1^2
&=\cot^2\theta_1
=\frac{1+\cos 2\theta_1}{1-\cos 2\theta_1}
=\frac{1+\sqrt{m_2/s}}{1-\sqrt{m_2/s}},
\\
\lambda_4^2
&=\cot^2\Big(\theta_1+\frac{3\pi}{4}\Big)
=\frac{1+\sin 2\theta_1}{1-\sin 2\theta_1}
=\frac{1+\sqrt{m_1/s}}{1-\sqrt{m_1/s}}.
\end{align*}
If we set
\[
p:=\max\{m_1,m_2\},
\qquad
q:=\min\{m_1,m_2\}=s-p,
\]
then
\[
a^2=\frac{1+\sqrt{p/s}}{1-\sqrt{p/s}}.
\]

The rough estimate \eqref{eq:roughg4} gives the range $q\ge 5$ as follows. Condition \eqref{eq:aineq} is equivalent to
\[
\sqrt{\frac{p}{s}}<\frac{\frac45(s-1)-1}{\frac45(s-1)+1},
\]
and hence to
\begin{equation}\label{eq:qineq}
q=s-p>s\left[1-\left(\frac{\frac45(s-1)-1}{\frac45(s-1)+1}\right)^2\right]
=\frac{80s(s-1)}{(4s+1)^2}.
\end{equation}
Now,
\[
5(4s+1)^2-80s(s-1)=120s+5>0,
\]
so
\[
\frac{80s(s-1)}{(4s+1)^2}<5.
\]
Thus \eqref{eq:qineq} is certainly satisfied if $q\ge 5$, i.e. if
\[
\min\{m_1,m_2\}\ge 5.
\]

To handle the borderline case $q=4$, we keep the term $B^2$ instead of discarding it. Since $2C^2\ge 0$,
\[
\Delta\ge (2n-2)-2S_x-2S_\nu+B^2-AB.
\]
Completing the square yields
\[
B^2-AB=\left(B-\frac{A}{2}\right)^2-\frac{A^2}{4}\ge -\frac{A^2}{4}.
\]
Moreover, since $\sum_i x_i^2=1$, the Cauchy--Schwarz inequality gives
\[
A^2=\left(\sum_i x_i^2\lambda_i\right)^2\le \sum_i x_i^2\lambda_i^2=S_x.
\]
Therefore
\begin{equation}\label{eq:refinedg4}
\Delta\ge (2n-2)-\tfrac{9}{4}\,S_x-2S_\nu
\ge (2n-2)-\tfrac{17}{4}\,a^2.
\end{equation}
Hence it is enough to require
\begin{equation}\label{eq:refinedaineq}
a^2<\frac{4}{17}(2n-2)=\frac{16}{17}(s-1).
\end{equation}
If $q=4$, then $p=s-4$, so
\[
a^2=\frac{1+\sqrt{1-4/s}}{1-\sqrt{1-4/s}}.
\]
Substituting this into \eqref{eq:refinedaineq} and simplifying, we obtain
\[
16s^2-304s-1<0.
\]
For integers $s$, this holds exactly when $s\le 19$: the left-hand side is increasing for $s\ge 10$, equals $-1$ at $s=19$, and equals $319$ at $s=20$. Therefore \eqref{eq:refinedg4} yields $\Delta>0$ whenever $q=4$ and $s\le 19$. On the other hand, for $q=4$ and $s\ge 15$, Proposition~\ref{prop:g4opt} gives
\[
\Delta(X,\nu)\ge 4s-4-4a^2>0.
\]
Since the ranges $s\le 19$ and $s\ge 15$ cover all $s\ge 8$, the ACS condition holds for every multiplicity pair with $q=4$.

This proves the ACS condition in all cases listed in Theorem \ref{thm:main}. The index estimate then follows immediately from Theorem \ref{thm:ACS}.
\end{proof}

\begin{remark}
The argument above leaves the following cases unresolved by this method: the case $g=3$ with multiplicity $m=2$, and the case $g=4$ when at least one multiplicity belongs to $\{2,3\}$. The cases $g=3$ with $m=1$ and $g=6$ are excluded because the ambient Ricci curvature is not positive.
\end{remark}

\begin{remark}\label{rem:pointwise-threshold}
The bound $\min\{m_1,m_2\}\ge 4$ in Theorem~\ref{thm:main}(2) is sharp with respect to the pointwise method. Indeed, the principal curvature whose absolute value equals~$a$ has multiplicity $q:=\min\{m_1,m_2\}$. Since positive Ricci curvature forces $q\ge 2$ in the case $g=4$, the corresponding eigenspace has dimension $\ge 2$, and we may choose orthogonal unit vectors $X,\nu$ supported entirely in that eigenspace. For such a choice,
\[
C=0,
\qquad
S_x=S_\nu=a^2,
\qquad
A=B=\lambda \quad (\text{where } |\lambda|=a),
\]
and the terms $B^2-AB$ and $2C^2$ all vanish. Proposition \ref{prop:Delta} therefore gives
\[
\Delta(X,\nu)=(2n-2)-4a^2.
\]
Since pointwise positivity requires $\Delta(X,\nu)>0$ for \emph{every} orthogonal unit pair $(X,\nu)$, and the configuration above attains the value $(2n-2)-4a^2$, such positivity forces $(2n-2)-4a^2>0$. Substituting
\[
a^2=\frac{1+\sqrt{1-q/s}}{1-\sqrt{1-q/s}},
\qquad s=m_1+m_2,
\]
and simplifying, this condition is equivalent to $q>4-4/s$, and hence to $q\ge 4$ for integer~$q$ and $s\ge 4$. Thus $\min\{m_1,m_2\}\ge 4$ is a \emph{necessary} condition for any argument based only on pointwise positivity of~$\Delta$.

In the present paper, this necessary condition is also shown to be sufficient. Indeed, the completed-square estimate \eqref{eq:refinedg4} covers the range $8\le s\le 19$, while Proposition~\ref{prop:g4opt} proves that for $s\ge 15$ the moment-relaxation bound $F\le 4a^2$ is achieved precisely at the extremal configuration above. Since
\[
4(s-1)-4a^2=\frac{8}{1+\sqrt{1-4/s}}>0,
\]
the argument closes for every $s\ge 8$, i.e. for every multiplicity pair with $q=4$.
\end{remark}

\end{document}